\documentclass[12pt,article]{amsart}
\usepackage{amssymb}
\usepackage{amsthm}
\usepackage{amsmath}

\usepackage{geometry}
\newtheorem{thm}{Theorem}
\newtheorem{lem}[thm]{Lemma}
\newtheorem{claim}{Claim}
\newtheorem{prop}[thm]{Proposition}
\newtheorem{cor}[thm]{Corollary}

\newcommand\diam{\textup{diam}\,}
\newcommand\diag{\textup{diag}\,}
\newcommand\rng{\textup{rng}\,}
\newcommand\ran{\textup{rng}\,}
\newcommand\tr{\textup{Tr}\,}
\newcommand\re{\textup{Re}}
\newcommand\im{\textup{Im}}

\newcommand{\C}{\mathbb{C}}

\newcommand{\R}{\mathbb{R}}

\newcommand{\T}{\mathbb{T}}
\newcommand{\tg}{\mathrm{tg}\,}
\newcommand{\ctg}{\mathrm{ctg}\,}

\newcommand{\irB}{\mathcal{B}}

\newcommand{\irH}{\mathcal{H}}

\newcommand{\irP}{\mathcal{P}}

\newcommand{\irS}{\mathcal{S}}

\newcommand{\irU}{\mathcal{U}}

\begin{document}

\title[Maps on classes of Hilbert space operators preserving measure of comm.]{Maps on classes of Hilbert space operators preserving measure of commutativity}
\author{Gy\"orgy P\'al Geh\'er}
\address{Bolyai Institute, University of Szeged, H-6720 Szeged, Aradi v\'ertan\'uk tere 1, Hungary}
\address{MTA-DE "Lend\"ulet" Functional Analysis Research Group, Institute of Mathematics, University of Debrecen, H-4010 Debrecen, P.O.~Box 12, Hungary}
\email{gehergy@math.u-szeged.hu}

\author{Gerg\H o Nagy}
\address{MTA-DE "Lend\" ulet" Functional Analysis Research Group, Institute of Mathematics, University of Debrecen, H-4010 Debrecen, P.O.~Box 12, Hungary}
\email{nagyg@science.unideb.hu}

\begin{abstract}
In this paper first we give a partial answer to a question of L. Moln\'ar and W. Timmermann. Namely, we will describe those linear (not necessarily bijective) transformations on the set of self-adjoint matrices which preserve a unitarily invariant norm of the commutator. After that we will characterize those (not necessarily linear or bijective) maps on the set of self-adjoint rank-one projections acting on a two-dimensional complex Hilbert space which leave the latter quantity invariant. Finally, this result will be applied in order to obtain a description of such bijective preservers on the unitary group and on the set of density operators.
\end{abstract}

\maketitle


\begin{center}
AMS (2010): Primary: 47B49. Secondary: 15A86.

\end{center}


\section{Introduction and statement of the results}

The relation of commutativity appears in most fields of mathematics and therefore the investigation of commutativity preserving transformations is a relevant problem. Such preservers on a certain class of operators are extremely important because they are connected to quantum mechanics. Namely, in the mathematical formalism of quantum mechanics a complex (and in most cases separable) Hilbert space can be associated to every quantum system. The so called observables correspond to self-adjoint operators, the pure states or rays are identified with self-adjoint rank-one projections, the mixed states are represented by density operators. The commutativity of these representing operators has a certain physical meaning.

The structure of mappings that preserve commutativity (usually in both directions) was investigated in many papers for different classes of operators, see for instance: \cite{CJR,MS,N,Se,S,U}. For several classes of normal operators it turned out that such bijections send each element -- up to unitary or antiunitary equivalence -- into a certain bounded Borel function of it. However, it is important to note that usually these results are valid only for Hilbert spaces with at least three dimensions. For example, in a two-dimensional space two self-adjoint operators commute if and only if they are linearly dependent or there exists a real-linear combination of them which equals the identity operator. Therefore in two dimensions many transformations exist on the set of self-adjoint operators which preserve commutativity in both directions.

However, if we pose a stronger condition on our transformation rather than simply the preservation of commutativity in both directions, we shall obtain more regular forms. One natural possibility is to consider the operator norm of the commutator and investigate such transformations that preserve this quantity. Concerning this kind of preservers, recently, L.~Moln\'ar and W.~Timmermann proved a theorem which is stated below, but before that we give some auxiliary definitions. Let $\irH$ denote a complex and at least two-dimensional Hilbert space. The symbols $\irB(\irH), \irB_s(\irH), \irP_1(\irH), \irU(\irH), \irS(\irH)$ will denote the set of bounded linear operators, bounded self-adjoint operators, self-adjoint rank-one projections, unitary operators and density operators acting on $\irH$, respectively. We note that a positive operator $A$ is said to be a density operator if $\tr A = 1$ where $\tr$ stands for the trace. The operator norm of an element $A\in\irB(\irH)$ will be denoted by $\|A\|$, and the vector norm of a vector $h\in\irH$ by $\|h\|$. A norm $|||\cdot|||$ on $\irB(\irH)$ is called unitarily invariant if $|||UAV||| = |||A|||$ is satisfied whenever $A\in\irB(\irH)$ and $U,V\in\irU(\irH)$. The reader can find a characterization of all unitarily invariant norms on matrices in \cite[Section IV.2.]{Bha}, which will be used many times throughout the paper. The commutator of two operators $A,B$ is the operator $AB-BA$ which is usually denoted by $[A,B]$. The previously mentioned result of Moln\'ar and Timmermann reads as follows.

\smallskip\smallskip

\noindent{\bf Theorem} (L. Moln\'{a}r and W. Timmermann \cite{MT}) \emph{Let $\irH$ be separable with $\dim \irH>2$. Assume $\phi\colon \irB_s(\irH) \to \irB_s(\irH)$ is a bijection such that}
\[
\big\|[\phi(A),\phi(B)]\big\| = \big\|[A,B]\big\|\quad(A,B \in \irB_s(\irH)).
\]
\emph{Then there exists either a unitary or an antiunitary operator $U$ on $\irH$, and functions $f\colon \irB_s(\irH) \to \R$ and $\tau\colon \irB_s(\irH)\to \{-1, 1\}$ such that}
\[
\phi(A) = \tau(A)UAU^* + f(A)I\quad(A \in \irB_s(\irH)).
\]

\smallskip\smallskip

At the end of their paper the authors point out that the question whether the same conclusion holds in a two-dimensional space still remains open. They also mention that in this case a characterization for linear and bijective preservers $\phi$ was found. However, they did not publish it because of the length of the proof and its extensive use of computation. Furthermore, from the characterization it was not clear whether such preservers have the same structure as in the above theorem. Another question which arises naturally concerns the conclusion of the above theorem if we drop the bijectivity condition or we replace the operator norm with some unitarily invariant norm. Their technique cannot be applied in these cases, since it uses \cite[Corollary 2]{MS2} which is valid only if the dimension is at least three and the transformation is bijective.

In the present paper first we intend to contribute to these questions. We will describe those linear transformations on $\irB_s(\irH)$ for finite-dimensional spaces $\irH$ which are not necessarily bijective and preserve a given unitarily invariant norm of the commutator. We note that the proof in the two-dimensional case is much more complicated, since in higher dimensions we can use a theorem of \cite{CJR}, but in two dimensions we have to develop a new technique.

\begin{thm}\label{self-adjointlinear2dimthm}
Suppose that $\dim\irH<\infty$, $|||\cdot|||$ is an arbitrary unitarily invariant norm and $\phi\colon \irB_s(\irH) \to \irB_s(\irH)$ is a (real-)linear transformation such that
\begin{equation}\label{E:noc}
\big|\big|\big|[\phi(A),\phi(B)]\big|\big|\big|=\big|\big|\big|[A,B]\big|\big|\big| \quad (A,B\in\irB_s(\irH)).
\end{equation}
Then there exists either a unitary or an antiunitary operator $U$
on $\irH$ and a linear functional $f\colon \irB_s(\irH) \to \R$ such that
\[
\phi(A) = UAU^* + f(A)I \quad(A \in \irB_s(\irH))
\]
or
\[
\phi(A) = -UAU^* + f(A)I \quad(A \in \irB_s(\irH)).
\]
\end{thm}

We would like to point out that in the above theorem the form of $\phi$ is global. We remark that if in this paper we consider a two-dimensional space, we usually identify it with $\C^2$ and the linear operators on $\C^2$ with $2\times 2$ complex matrices in the natural way. The following result describes maps on $\irP_1(\C^2)$ which preserve a given unitarily invariant norm of the commutator. It will play a crucial role in the proof of Theorems \ref{T:unit} and \ref{T:dense}.

\begin{thm}\label{T:proj}
Assume that $|||\cdot|||$ is an arbitrary unitarily invariant norm. Let $\phi\colon \irP_1(\C^2)\to \irP_1(\C^2)$ be a map for which
\begin{equation}\label{E:nocp}
\big|\big|\big|[\phi(P),\phi(Q)]\big|\big|\big|=\big|\big|\big|[P,Q]\big|\big|\big|\quad(P,Q\in \irP_1(\C^2))
\end{equation}
is satisfied. Then there exists a unitary or an antiunitary operator $U$ on $\C^2$ such that for each $P\in \irP_1(\C^2)$ we have
\begin{equation}\label{E:projchoice}
\phi(P) \in \{UPU^*, UP^{\bot}U^*\}.
\end{equation}
\end{thm}

We would like to clarify that in Theorem \ref{T:proj} the choice of $\phi(P)$ can really vary elementwise. In fact, it is quite easy to see that every map $\phi$ which has the form \eqref{E:projchoice} satisfies \eqref{E:nocp}. A counterpart of the above theorem in higher dimensions will be proven in the next section as Proposition \ref{P:proj3}. A reformulation of Theorem \ref{T:proj} will be also given as Corollary \ref{C:proj}.

The structure of commutativity preserving maps on unitary groups was determined in \cite{MS}. Our next result concerns transformations on $\irU(\irH)$ that preserve the norm of the commutator.

\begin{thm}\label{T:unit}
Suppose that $2\leq \dim\irH \leq \aleph_0$ and $|||\cdot|||$ is an arbitrary unitarily invariant norm. Let $\phi\colon\irU(\irH)\to\irU(\irH)$ be a bijection such that
\begin{equation}\label{E:pres}
\big|\big|\big|[\phi(V_1),\phi(V_2)]\big|\big|\big|=\big|\big|\big|[V_1,V_2]\big|\big|\big|\quad(V_1,V_2\in\irU(\irH))
\end{equation}
holds. Then there exist a unitary or an antiunitary operator $U$ on $\irH$ and a function $\tau\colon\irU(\irH)\to\mathbb{T}$ such that for each $V\in\irU(\irH)$ we have
\[
\phi(V) \in \{ \tau(V)UVU^*, \tau(V)UV^*U^* \}.
\]
Moreover, if $\dim\irH = 2$, then the bijectivity condition can be relaxed to surjectivity.
\end{thm}

The second author described the structure of those bijections on $\irS(\irH)$ which preserve a special unitarily invariant norm of the commutator when $\dim\irH>2$ (see \cite[Theorem 2 and 3]{N}). Our final result is an extension of that and it reads as follows.

\begin{thm}\label{T:dense}
\begin{itemize}
\item[(a)] Suppose that $2\leq\dim\irH<\infty$ and $|||\cdot|||$ is an arbitrary unitarily invariant norm. Let $\phi\colon\irS(\irH)\to\irS(\irH)$ be a bijection with the following property:
\begin{equation}\label{E:pres_dense}
\big|\big|\big|[\phi(A),\phi(B)]\big|\big|\big|=\big|\big|\big|[A,B]\big|\big|\big| \quad (A,B\in\irS(\irH)).
\end{equation}
Then there is a unitary or an antiunitary operator $U$ on $\irH$ such that
\[
\phi(A) \in \left\{ UAU^*, \frac{2}{\dim\irH}I-UAU^*\right\} \quad (A\in\irS(\irH)).
\]
Moreover, if $\dim\irH = 2$, then the bijectivity condition can be relaxed to surjectivity.

\item[(b)] Suppose that $\dim\irH=\infty$ and $\phi\colon\irS(\irH)\to\irS(\irH)$ is a bijection such that \eqref{E:pres_dense} is valid. Then there is a unitary or an antiunitary operator $U$ on $\irH$ such that
\[
\phi(A)=UAU^*\quad(A\in\irS(\irH)).
\]
\end{itemize}
\end{thm}

In Theorems \ref{T:unit} and \ref{T:dense} we assume bijectivity if the dimension is greater than two. The reason is that in the proof we present such techniques which use the bijectivity property of the mapping.


\section{Proofs}\label{S:P}

We begin with the proof of the modified linear Moln\'ar-Timmermann preserver problem.

\begin{proof}[Proof of Theorem \ref{self-adjointlinear2dimthm}]
According to the dimension of $\irH$ we divide our argument into two cases.

\smallskip\smallskip

\noindent CASE I. \emph{When $\dim\irH>2$.} \\ Since $\phi$ preserves commutativity in both directions, by \cite[Theorem 2]{CJR} we obtain that there is a unitary or an antiunitary operator $U$ on $\irH$, a linear functional $f\colon\irB_s(\irH)\to\R$ and a number $c\in\R$ such that
\[
\phi(A)=cUAU^*+f(A)I\quad(A\in\irB_s(\irH))
\]
holds. A simple calculation shows that the transformation $A\mapsto U^*AU$ preserves the quantity $\big|\big|\big|[A,B]\big|\big|\big|$ $(A,B\in\irB_s(\irH))$. Since this is true also for $\phi$, it holds for their composition which is precisely the mapping $A\mapsto cA+f(A)I\ (A\in\irB_s(\irH))$. We deduce that
\[
c^2\big|\big|\big|[A,B]\big|\big|\big|=\big|\big|\big|[A,B]\big|\big|\big|\quad(A,B\in\irB_s(\irH))
\]
is valid and therefore $c\in\{-1,1\}$. Hence we conclude that $\phi$ is of the desired form.

\smallskip\smallskip

\noindent CASE II. \emph{When $\dim\irH=2$.} \\ Here, the main idea is as follows. Let $A,B\in\irB_s(\C^2)$. It can be observed that $[A,B]$ is skew-Hermitian (hence normal) and that $\tr [A,B] = 0$. Consequently, the eigenvalues of $[A,B]$ are $\pm i\sqrt{\det{[A,B]}}$, and therefore we easily obtain that $\big|\big|\big|[A,B]\big|\big|\big| = c\sqrt{\det{[A,B]}}$ holds with a number $c>0$ (which is independent of $A$ and $B$, see \cite[Section IV.2.]{Bha}). This implies that we have
\begin{equation}\label{E:detcom}
\det[\phi(A),\phi(B)]=\det[A,B] \quad (A,B\in\irB_s(\C^2))
\end{equation}
if and only if \eqref{E:noc} is satisfied. This observation will significantly simplify the proof and it will be used throughout it.

Let us observe that $\phi$ preserves commutativity in both directions. Therefore $\phi(\irB_s(\C^2)))$ is not commutative, which implies that $\phi(I) = \varphi I$ is valid with some $\varphi \in \R$. Next, by considering the map
\[
\phi_1(\cdot) = U\phi(\cdot)U^*
\]
where $U$ is an appropriate unitary matrix, we see that $\phi_1$ obviously satisfies \eqref{E:detcom}, moreover
\[
\phi_1\left(
\begin{matrix}
1 & 0 \\
0 & 0
\end{matrix}
\right) = \left(
\begin{matrix}
s & 0 \\
0 & t
\end{matrix}
\right) =: A_1
\]
holds with some $s,t\in\R$. Of course $\phi_1(I) = \varphi I$ is valid as well. Using the linearity of $\phi_1$, we obtain
\[
\phi_1\left(
\begin{matrix}
0 & 0 \\
0 & 1
\end{matrix}
\right) = \varphi I - \left(
\begin{matrix}
s & 0 \\
0 & t
\end{matrix}
\right) =: A_2.
\]
Set
\[
\phi_1\left(
\begin{matrix}
0 & 1 \\
1 & 0
\end{matrix}
\right) = \left(
\begin{matrix}
a & u \\
\overline{u} & d
\end{matrix}
\right) =: A_3
\]
and
\[
\phi_1\left(
\begin{matrix}
0 & i \\
-i & 0
\end{matrix}
\right) = \left(
\begin{matrix}
e & w \\
\overline{w} & h
\end{matrix}
\right) =: A_4,
\]
where $a,d,e,h\in\R$, $u,w\in\C$, and let
\[
b=\re u,\ c=\im u,\ f=\re w,\ g=\im w.
\]
From the above definitions we infer that
\[
\phi_1\left(
\begin{matrix}
\alpha & \beta + i\gamma \\
\beta - i\gamma & \delta
\end{matrix}
\right) = \alpha A_1 + \delta A_2 + \beta A_3 + \gamma A_4
\]
is satisfied for any $\alpha,\beta,\gamma,\delta \in \R$.

Our strategy is that we will consider the equation \eqref{E:detcom} in four special cases. This will give us some information about $\phi_1$. First, let us write the equality below which is satisfied by every number $\alpha, \beta, \gamma,\delta\in\R$:
\[
|\beta + i\gamma|^2(\alpha-\delta)^2 = \det \left[
\left(
\begin{matrix}
\alpha & 0 \\
0 & \delta
\end{matrix}
\right),
\left(
\begin{matrix}
0 & \beta + i\gamma \\
\beta - i\gamma & 0
\end{matrix}
\right) \right]
\]
\[
= \det \left[
\alpha A_1 + \delta A_2, \beta A_3 + \gamma A_4
\right] = (s - t)^2 |u \beta + w \gamma|^2 (\alpha - \delta)^2.
\]
It follows that
\[
|\beta + i\gamma| = |s - t||u \beta + w \gamma| \quad(\beta,\gamma\in\R).
\]
Hence $s\neq t$ and the map $z\mapsto|s - t|(u\re z+w\im z)\ (z\in\C)$ is a real-linear isometry. Therefore $w = \pm i u$ and consequently $\overline{u}w + u\overline{w} = 0$ holds, moreover,
\begin{equation}\label{E:norm}
|u|=|w|=\frac{1}{|s - t|}.
\end{equation}

Next, applying \eqref{E:norm} and the orthogonality of $u$ and $w$, we deduce that
\[
4 = \det \left[\left(
\begin{matrix}
0 & 1 \\
1 & 0
\end{matrix}\right),
\left(\begin{matrix}
0 & i \\
-i & 0
\end{matrix}\right)\right] = \det[A_3,A_4]
\]
\[
= \det\left(\begin{matrix}
-w \overline{u} + u \overline{w} & -e u + h u + a w - d w \\
e \overline{u} - h \overline{u} - a \overline{w} + d \overline{w} & w \overline{u} - u \overline{w}
\end{matrix}\right)
\]
\[
= -(w \overline{u} - u \overline{w})^2+|e u - h u - a w + d w|^2 = \frac{4}{(s-t)^4} + |e u - h u - a w + d w|^2
\]
\[
= \frac{4}{(s-t)^4} + (e-h)^2 |u|^2 - (a-d)(e-h)(w \overline{u} + u \overline{w}) + (a-d)^2 |w|^2
\]
\[
= \frac{4}{(s-t)^4} + \frac{(e-h)^2+(a-d)^2}{(s-t)^2},
\]
whence we obtain
\begin{equation}\label{E:E1}
4(s-t)^2 - \frac{4}{(s-t)^2} = (e-h)^2+(a-d)^2.
\end{equation}

In the third case, using also \eqref{E:E1} we compute the following:
\[
5 = \det \left[\left(
\begin{matrix}
1 & 1 \\
1 & 0
\end{matrix}\right),
\left(\begin{matrix}
0 & i \\
-i & 0
\end{matrix}\right)\right] = \det[A_1+A_3,A_4]
\]
\[
= \det\left(
\begin{matrix}
-w \overline{u} + u \overline{w} & -e u + h u + (a + s) w - (d + t) w\\
e \overline{u} - h \overline{u} - (a + s) \overline{w} + (d + t) \overline{w} & w \overline{u} - u \overline{w}
\end{matrix}
\right)
\]
\[
= \frac{4}{(s-t)^4} + (e-h)^2 |u|^2 + (a-d)^2 |w|^2 + 2 (a-d) (s-t)|w|^2 + (s-t)^2 |w|^2
\]
\[
= \frac{4}{(s-t)^4} + \frac{(e-h)^2 + (a-d)^2 + 2 (a-d) (s-t) + (s-t)^2}{(s-t)^2}
\]
\[
= \frac{4}{(s-t)^4} + \frac{5(s-t)^2 - \frac{4}{(s-t)^2} + 2 (a-d) (s-t)}{(s-t)^2} = 5 + 2 \frac{a-d}{s-t},
\]
which implies $a = d$.

Finally, very similarly we obtain the following:
\[
5 = \det \left[\left(
\begin{matrix}
0 & 1 \\
1 & 0
\end{matrix}\right),
\left(\begin{matrix}
1 & i \\
-i & 0
\end{matrix}\right)\right] = \det[A_3,A_1+A_4]
\]
\[
= \det\left(
\begin{matrix}
-w \overline{u} + u \overline{w} & -(e + s) u + (h + t) u + a w - d w\\
(e + s) \overline{u} - (h + t) \overline{u} - a \overline{w} + d \overline{w} & w \overline{u} - u \overline{w}
\end{matrix}
\right)
\]
\[
= \frac{4}{(s-t)^4} + (e-h)^2 u \overline{u} + 2 (e-h) (s-t) u \overline{u} + (s-t)^2 u \overline{u} + (a-d)^2 w \overline{w} = 5 + 2 \frac{e-h}{s-t},
\]
hence $e = h$ follows. Using what we have shown in the last three cases, we conclude that
\[
4(s-t)^2 - \frac{4}{(s-t)^2} = (e-h)^2+(a-d)^2 = 0,
\]
therefore $s-t = \pm 1$, which -- by \eqref{E:norm} -- yields that $|u|=|w|=1$.

Now for $\alpha,\beta,\gamma,\delta\in\R$ we define the following matrices:
\[
N_{+}(\alpha,\beta,\gamma,\delta) = \alpha \left(
\begin{matrix}
0 & 0 \\
0 & 1
\end{matrix}
\right) + \delta
\left(\begin{matrix}
1 & 0 \\
0 & 0
\end{matrix}\right)
 + \beta
\left(\begin{matrix}
0 & u \\
\overline{u} & 0
\end{matrix}\right)+ \gamma
\left(\begin{matrix}
0 & w \\
\overline{w} & 0
\end{matrix}\right)
\]
and
\[
N_{-}(\alpha,\beta,\gamma,\delta) = \alpha \left(
\begin{matrix}
0 & 0 \\
0 & -1
\end{matrix}
\right) + \delta
\left(\begin{matrix}
-1 & 0 \\
0 & 0
\end{matrix}\right)
 + \beta
\left(\begin{matrix}
0 & u \\
\overline{u} & 0
\end{matrix}\right)+ \gamma
\left(\begin{matrix}
0 & w \\
\overline{w} & 0
\end{matrix}\right)
\]
We also define the linear functional
\[
f\colon \irB_s(\C^2)\to\R, \quad
f\left(
\begin{matrix}
\alpha & \beta + i\gamma \\
\beta - i\gamma & \delta
\end{matrix}
\right) = (\alpha s + \beta a + \gamma e - \delta (t - \varphi))I.
\]
By what we have proven so far, the relation
\begin{equation}\label{E:MT2DPaf}
\begin{gathered}
\phi_1\left(
\begin{matrix}
\alpha & \beta + i\gamma \\
\beta - i\gamma & \delta
\end{matrix}
\right) =
\alpha \left(
\begin{matrix}
s & 0 \\
0 & t
\end{matrix}
\right) +\delta\varphi I - \delta
\left(\begin{matrix}
s & 0 \\
0 & t
\end{matrix}\right)
 + \beta
\left(\begin{matrix}
a & u \\
\overline{u} & a
\end{matrix}\right)
+ \gamma
\left(\begin{matrix}
e & w \\
\overline{w} & e
\end{matrix}\right) \\
\in\left\{N_+(\alpha,\beta,\gamma,\delta) + f\left(
\begin{matrix}
\alpha & \beta + i\gamma \\
\beta - i\gamma & \delta
\end{matrix}
\right)I,
N_-(\alpha,\beta,\gamma,\delta) + f\left(
\begin{matrix}
\alpha & \beta + i\gamma \\
\beta - i\gamma & \delta
\end{matrix}
\right)I\right\}
\end{gathered}
\end{equation}
is valid where $|u|=1$ and $w \in \{-iu, iu\}$ (and obviously $u$ and $w$ are independent of the actual value of $\alpha,\beta,\gamma,\delta$). Now observe that
\[
N_+(\alpha,\beta,\gamma,\delta) = \left\{\begin{matrix}
\overline{\left(\begin{matrix}
0 & \overline{u} \\
1 & 0
\end{matrix}\right)
\left(
\begin{matrix}
\alpha & \beta + i\gamma \\
\beta - i\gamma & \delta
\end{matrix}
\right)
\left(\begin{matrix}
0 & \overline{u} \\
1 & 0
\end{matrix}\right)^*}  & \text{if } w = iu \\
 & \\
\left(\begin{matrix}
0 & u \\
1 & 0
\end{matrix}\right)
\left(
\begin{matrix}
\alpha & \beta + i\gamma \\
\beta - i\gamma & \delta
\end{matrix}
\right)
\left(\begin{matrix}
0 & u \\
1 & 0
\end{matrix}\right)^*  & \text{if } w = -iu
\end{matrix}\right.
\]
and
\[
N_-(\alpha,\beta,\gamma,\delta) = \left\{\begin{matrix}
-\overline{\left(\begin{matrix}
0 & -\overline{u} \\
1 & 0
\end{matrix}\right)
\left(
\begin{matrix}
\alpha & \beta + i\gamma \\
\beta - i\gamma & \delta
\end{matrix}
\right)
\left(\begin{matrix}
0 & -\overline{u} \\
1 & 0
\end{matrix}\right)^*} & \text{if } w = iu\\
 & \\
-\left(\begin{matrix}
0 & -u \\
1 & 0
\end{matrix}\right)
\left(
\begin{matrix}
\alpha & \beta + i\gamma \\
\beta - i\gamma & \delta
\end{matrix}
\right)
\left(\begin{matrix}
0 & -u \\
1 & 0
\end{matrix}\right)^* & \text{if } w = -iu
\end{matrix}\right.
\]
where $\overline{\cdot}$ denotes elementwise conjugation.

From \eqref{E:MT2DPaf} and the observations above we get that
\[
\phi_1(A) \in\{ \psi_1(A) := U_1AU_1^* + f(A)I, \psi_2(A) := - U_2AU_2^* + f(A)I \}
\]
holds for every $A\in\irB_s(\C^2)$ where $U_1$ and $U_2$ are both unitary, or antiunitary operators. Let $H_j = \ker(\psi_j-\phi_1)$ ($j=1,2$). Since linear maps on $\irB_s(\C^2)$ are continuous, both $H_1$ and $H_2$ are closed sets and obviously $\irB_s(\C^2) = H_1\cup H_2$. By Baire's category theorem, one of them contains an open ball of $\irB_s(\C^2)$, and by linearity this set coincides with the whole space $\irB_s(\C^2)$. Consequently, we have
\[
\phi_1(A) = U_1AU_1^* + f(A)I \quad (A\in\irB_s(\C^2))
\]
or
\[
\phi_1(A) = -U_2AU_2^* + f(A)I \quad (A\in\irB_s(\C^2)).
\]
Transforming back to our original mapping $\phi$, we easily complete the proof in the two-dimensional case.
\end{proof}


Before proving our result concerning the preserver problem on $\irP_1(\C^2)$, let us make some observations. Let $\mathcal{N}$ denote the set of those rank-two self-adjoint operators on $\irH$ whose spectrum contains $\{-1,1\}$. The spectrum of any operator $T$ will be denoted by $\sigma(T)$. For any $u,v\in\irH$ the symbol $u\otimes v$ will stand for the rank-one element of $\irB(\irH)$ defined by $(u\otimes v)y=\langle y,v\rangle u\ (y\in\irH)$. It is easy to see that there is a real number $c>0$ such that
\begin{equation}\label{E:c}
|||N|||=c\quad(N\in\mathcal{N}).
\end{equation}

Next, we show that if $A\in\irS(\irH)$ is a density operator and $x\in\irH$ is a unit vector, then we have
\begin{equation}\label{mcd}
\big|\big|\big|[A,x\otimes x]\big|\big|\big| = c\sqrt{\langle A^2x,x\rangle - \langle Ax,x\rangle^2},
\end{equation}
(observe that by the Cauchy-Schwarz inequality $\langle A^2x,x\rangle-\langle Ax,x\rangle^2\ge0$). In order to verify \eqref{mcd}, first we remark that when $x$ is an eigen-vector of $A$, then we have
\[
\big|\big|\big|[A,x\otimes x]\big|\big|\big| = \big|\big|\big|(Ax)\otimes x - x\otimes (Ax)\big|\big|\big| = 0 = c\sqrt{\langle A^2x,x\rangle - \langle Ax,x\rangle^2}.
\]
Second suppose that $x$ and $Ax$ are linearly independent and let $T = [A,x\otimes x]$. A straightforward calculation gives us the matrix of $T|_{\rng T}$ with respect to the basis $\{x,Ax\}$ (in this paper $\rng$ denotes the range of maps):
\[
\left(\begin{array}{cc}
-\langle Ax,x\rangle & -\langle A^2x,x\rangle \\
1 & \langle Ax,x\rangle
\end{array}\right).
\]
This implies that $\sigma(T)\setminus\{0\} = \{\pm i\sqrt{\langle A^2x,x\rangle-\langle Ax,x\rangle^2}\}$. By the spectral theorem and \eqref{E:c} we get \eqref{mcd}.

Throughout this section we will use the notation $P_u = u\otimes u \in\irP_1(\irH)$ for any unit vector $u\in\irH$. Moreover, if we write $P_u$, then it will always be implicitly assumed that $\|u\| = 1$. If we consider two elements $P_u, P_v\in\irP_1(\irH)$, then by applying \eqref{mcd} we get
\[
\big|\big|\big|[P_u,P_v]\big|\big|\big| = c\sqrt{\langle P_u v,v\rangle - \langle P_u v,v\rangle^2} = c\sqrt{|\langle u,v\rangle|^2 - |\langle u,v\rangle|^4} = c\sqrt{\tr P_uP_v-(\tr P_uP_v)^2}
\]
where the well-known equation $\tr P_uP_v = |\langle u,v\rangle|^2$ was used and will be used often in the proof of Theorem \ref{T:proj}. This shows that a mapping on $\irP_1(\C^2)$ satisfies \eqref{E:nocp} if and only if it leaves the quantity
\[
f(P,Q)=\tr PQ-(\tr PQ)^2
\]
invariant $(P,Q\in\irP_1(\C^2))$. We say that a mapping $\phi\colon \irP_1(\irH) \to \irP_1(\irH)$ has the property \eqref{E:*} if
\begin{equation}\tag{*}\label{E:*}
\tr \phi(P)\phi(Q)\in\{\tr PQ,1-\tr PQ\}\quad(P,Q\in\irP_1(\C^2)).
\end{equation}
By a straightforward calculation we see that $\phi$ has the property \eqref{E:*} if and only if it satisfies \eqref{E:nocp}.

A transformation $\phi\colon\irP_1(\C^2)\to\irP_1(\C^2)$ is called a locally polynomial map (or LPM, for short) if for every $P\in\irP_1(\C^2)$ we have $\phi(P)\in\{P,P^{\bot}\}$. An easy calculation shows that any LPM has the property \eqref{E:*}. Throughout this section $\diag(a_1,a_2) := \left(\begin{matrix}
a_1 & 0 \\
0 & a_2
\end{matrix}\right)$ $(a_1,a_2\in\mathbb{C})$. Now, we are in a position to prove Theorem \ref{T:proj}. We note that our proof includes three claims.

\begin{proof}[Proof of Theorem \ref{T:proj}]
Let us consider an injective map $\Phi\colon \irP_1(\C^2)\to \irP_1(\C^2)$ that satisfies \eqref{E:*} (the general case will be handled at the end of this proof). Let $P,Q\in\irP_1(\mathbb{C}^2)$. Observe that $f(P,Q)=0$ exactly when $[P,Q] = 0$. By the injectivity of $\Phi$, it follows that $\Phi$ preserves orthogonality in both directions, i.~e.~$\Phi(P)$ and $\Phi(Q)$ are orthogonal if and only if the same holds for $P$ and $Q$.

For any unitary matrix $U$, we see that the mapping
\[
\Phi_1(\cdot) = U\Phi(\cdot)U^*
\]
is obviously injective and satisfies \eqref{E:*}. We may choose such a $U$ for which
\begin{equation}\label{E:ass1}
\Phi_1\left(P_{(1,0)}\right) = P_{(1,0)}
\end{equation}
holds. Then $\Phi_1\left(P_{(0,1)}\right) = P_{(0,1)}$ follows immediately. Throughout the proof we will implicitly use the elementary fact that any unit vector in $\C^2$ is a scalar multiple of a vector $(\cos t,\lambda\sin t)$ with some numbers $t\in[0,\pi/2],\ \lambda\in\T$.

Let $t\in]0,\pi/2[$, $\lambda\in\T$ be arbitrary, and $\Phi_1\left(P_{(\cos t,\lambda\sin t)}\right) = P_{(w_1,w_2)}$ with some unit vector $(w_1,w_2)\in\C^2$. Since we have
\[
\tr P_{(1,0)}P_{(\cos t,\lambda\sin t)} = \cos^2 t,
\]
by \eqref{E:ass1} and \eqref{E:*} we infer
\[
|w_1|^2=\tr P_{(1,0)}P_{(w_1,w_2)}\in\{\sin^2 t,\ \cos^2 t\}.
\]
Therefore we conclude that
\[
\Phi_1\left(P_{(\cos t,\lambda\sin t)}\right) \in \left\{P_{(\cos t,\mu\sin t)},\ P_{(\sin t,\mu\cos t)}\colon \mu\in\T\right\}.
\]
We immediately get that $\Phi_1\left(P_{\left(1/\sqrt{2},1/\sqrt{2}\right)}\right) = P_{\left(1/\sqrt{2},\lambda_0/\sqrt{2}\right)}$ is valid with some $\lambda_0\in\T$. Let us consider the transformation
\[
\Phi_2(\cdot) = \diag(1,\lambda_0)^*\Phi_1(\cdot)\diag(1,\lambda_0)
\]
that is obviously injective, satisfies \eqref{E:*} and $\Phi_2(P_{(1,0)}) = P_{(1,0)}$, moreover, we have
\begin{equation}\label{E:ass2}
\Phi_2\left(P_{\left(\frac{1}{\sqrt{2}},\frac{1}{\sqrt{2}}\right)}\right) = P_{\left(\frac{1}{\sqrt{2}},\frac{1}{\sqrt{2}}\right)}.
\end{equation}
Then, since $\Phi_2$ preserves orthogonality we get that $\Phi_2\left(P_{\left(\frac{1}{\sqrt{2}},\frac{-1}{\sqrt{2}}\right)}\right) = P_{\left(\frac{1}{\sqrt{2}},\frac{-1}{\sqrt{2}}\right)}$. We proceed with the proof of the following claim.

\begin{claim}\label{C:1} We have either
\begin{equation}\label{lemma_case1}
\Phi_2\left(P_{\left(\frac{1}{\sqrt{2}},\frac{\sigma}{\sqrt{2}}\right)}\right) \in \left\{ P_{\left(\frac{1}{\sqrt{2}},-\frac{\sigma}{\sqrt{2}}\right)}, P_{\left(\frac{1}{\sqrt{2}},\frac{\sigma}{\sqrt{2}}\right)} \right\} \quad (\sigma\in\T\setminus\{1,-1\}),
\end{equation}
or
\begin{equation}\label{lemma_case2}
\Phi_2\left(P_{\left(\frac{1}{\sqrt{2}},\frac{\sigma}{\sqrt{2}}\right)}\right) \in \left\{ P_{\left(\frac{1}{\sqrt{2}},\frac{\overline{\sigma}}{\sqrt{2}}\right)}, P_{\left(\frac{1}{\sqrt{2}},-\frac{\overline{\sigma}}{\sqrt{2}}\right)}\right\} \quad (\sigma\in\T\setminus\{1,-1\}).
\end{equation}
\end{claim}

\begin{proof}
Set an arbitrary $\sigma\in\T\setminus\{1,-1\}$ and let $\Phi_2\left(P_{\left(\frac{1}{\sqrt{2}},\frac{\sigma}{\sqrt{2}}\right)}\right) = P_{\left(\frac{1}{\sqrt{2}},\frac{\tilde\sigma}{\sqrt{2}}\right)}$ with some $\tilde\sigma\in\T$. We have
\[
\tr P_{\left(\frac{1}{\sqrt{2}},\frac{1}{\sqrt{2}}\right)}P_{\left(\frac{1}{\sqrt{2}},\frac{\sigma}{\sqrt{2}}\right)} = \frac{|1+\overline{\sigma}|^2}{4}
\]
and
\[
\tr P_{\left(\frac{1}{\sqrt{2}},\frac{1}{\sqrt{2}}\right)}P_{\left(\frac{1}{\sqrt{2}},\frac{\tilde\sigma}{\sqrt{2}}\right)} = \frac{|1+\overline{\tilde\sigma}|^2}{4}.
\]
By \eqref{E:ass2}, \eqref{E:*} and the parallelogram law we obtain
\[
|1+\overline{\tilde\sigma}|^2 \in \big\{|1+\overline{\sigma}|^2, 4-|1+\overline{\sigma}|^2\big\} = \big\{|1+\overline{\sigma}|^2, |1-\overline{\sigma}|^2\big\}
\]
Since for a given number $\lambda\in\T$ we have $|1+\sigma| = |1+\lambda|$ if and only if $\lambda \in \{\sigma, \overline{\sigma}\}$, we infer
\[
\tilde\sigma\in\{ \sigma, -\sigma, \overline{\sigma},-\overline{\sigma}\} \quad (\sigma\in\T\setminus\{1,-1\}).
\]
In particular we get
\[
\Phi_2\left(P_{\left(\frac{1}{\sqrt{2}},\frac{i}{\sqrt{2}}\right)}\right) \in \left\{P_{\left(\frac{1}{\sqrt{2}},\frac{i}{\sqrt{2}}\right)},P_{\left(\frac{1}{\sqrt{2}},-\frac{i}{\sqrt{2}}\right)}\right\}.
\]

Finally define $\tilde{\T} = \T\setminus\{1,-1,i,-i\}$ and choose arbitrary numbers $\sigma_1,\sigma_2\in\tilde{\T}$. We have
\begin{equation}\label{E:C1vege1}
\tr P_{\left(\frac{1}{\sqrt{2}},\frac{\sigma_1}{\sqrt{2}}\right)}P_{\left(\frac{1}{\sqrt{2}},\frac{\sigma_2}{\sqrt{2}}\right)} = \frac{|1+\sigma_1\overline{\sigma_2}|^2}{4}
\end{equation}
and
\begin{equation}\label{E:C1vege2}
\tr P_{\left(\frac{1}{\sqrt{2}},\frac{\tilde{\sigma}_1}{\sqrt{2}}\right)}P_{\left(\frac{1}{\sqrt{2}},\frac{\tilde{\sigma}_2}{\sqrt{2}}\right)} = \frac{\big|1+\tilde{\sigma}_1\overline{\tilde{\sigma}_2}\big|^2}{4}.
\end{equation}
Suppose for a moment that $\tilde{\sigma_1}=\pm\sigma_1$ and $\tilde{\sigma_2}=\pm\overline{\sigma_2}$. Then, using the parallelogram law again, \eqref{E:C1vege1} and \eqref{E:C1vege2} yield that $\sigma_1\in\{-\overline{\sigma_1},\overline{\sigma_1}\}$ or $\sigma_2\in\{-\overline{\sigma_2},\overline{\sigma_2}\}$, which contradicts the condition $\sigma_1,\sigma_2\in\tilde{\T}$. Whence we conclude that $\tilde\sigma\in\{\sigma,-\sigma\}$ holds for all $\sigma\in\tilde{\T}$ or $\tilde\sigma\in\{\overline{\sigma},-\overline{\sigma}\}$ is satisfied by every $\sigma\in\tilde{\T}$.
\end{proof}

Now we define a new mapping as follows:
\[
\Phi_3(\cdot) = \left\{
\begin{matrix}
\Phi_2(\cdot) & \text{if } \eqref{lemma_case1} \text{is valid,}\\
K\Phi_2(\cdot)K & \text{if } \eqref{lemma_case2} \text{is valid,}\\
\end{matrix}
\right.
\]
where $K$ denotes the coordinatewise conjugation operator which is antiunitary. Trivially $\Phi_3$ is injective, it has the property \eqref{E:*}, it satisfies $\Phi_3(P_{(1,0)}) = P_{(1,0)}$ and
\begin{equation}\label{E:Phi3}
\Phi_3\left(P_{\left(\frac{1}{\sqrt{2}},\frac{\sigma}{\sqrt{2}}\right)}\right) \in \left\{ P_{\left(\frac{1}{\sqrt{2}},-\frac{\sigma}{\sqrt{2}}\right)}, P_{\left(\frac{1}{\sqrt{2}},\frac{\sigma}{\sqrt{2}}\right)} \right\} \quad (\sigma\in\T\setminus\{1,-1\}).
\end{equation}

Next we establish the following important information about the general form of $\Phi_3$.

\begin{claim}\label{C:2}
We have
\[
\Phi_3(P_{(\cos t, \nu\sin t)}) \in \left\{P_{(\cos t, \nu\sin t)}, P_{(\sin t, -\nu\cos t)}, P_{(\sin t, \nu\cos t)}, P_{(\cos t, -\nu\sin t)}\right\} \quad (\nu\in\T, t\in]0,\pi/2[).
\]
\end{claim}

\begin{proof}
Set $\nu\in\T$ and $t\in\left]0,\frac{\pi}{2}\right[$. We have learnt that
\[
\Phi_3\left(P_{(\cos t,\nu\sin t)}\right) \in \{P_{(\cos t,\lambda_{\nu,t}\sin t)}, P_{(\sin t,\lambda_{\nu,t}\cos t)}\}
\]
holds with an appropriate $\lambda_{\nu,t}\in\T$. We have the following equalities:
\[
\tr P_{(\frac{1}{\sqrt{2}},\frac{\nu}{\sqrt{2}})}P_{(\cos t,\nu\sin t)} = \frac{\big|\cos t + \sin t\big|^2}{2},
\]
\[
\tr P_{(\frac{1}{\sqrt{2}},\frac{\pm\nu}{\sqrt{2}})}P_{(\cos t,\lambda_{\nu,t}\sin t)} = \frac{\big|\cos t \pm \nu\overline{\lambda_{\nu,t}} \sin t\big|^2}{2},
\]
\[
\tr P_{(\frac{1}{\sqrt{2}},\frac{\pm\nu}{\sqrt{2}})}P_{(\sin t,\lambda_{\nu,t}\cos t)} = \frac{\big|\sin t \pm \nu\overline{\lambda_{\nu,t}} \cos t\big|^2}{2}=\frac{\big|\cos t \pm \overline{\nu}\lambda_{\nu,t} \sin t\big|^2}{2}.
\]
Hence we conclude by \eqref{E:*}, \eqref{E:Phi3} and the parallelogram law that $\nu\overline{\lambda_{\nu,t}}\in\{1,-1\}$ holds, and consequently: $\lambda_{\nu,t}\in\{\nu,-\nu\}$.
\end{proof}

We proceed with the proof of the forthcoming claim.

\begin{claim}\label{C:3}
Our original transformation $\Phi(\cdot)$ can be written in the form $UL(\cdot)U^*$ with a unitary or an antiunitary operator $U$ on $\C^2$ and an injective LPM $L \colon\irP_1(\C^2)\to\irP_1(\C^2)$.
\end{claim}

\begin{proof}
Clearly, if we prove the statement for $\Phi_3$, then by transforming back to $\Phi$ we conclude that $\Phi$ has the above form. Therefore we will investigate $\Phi_3$.

The composition of an injective LPM $L$ and $\Phi_3$ is clearly injective and has the property \eqref{E:*}. By the previous claim, we can choose such an $L$ for which the mapping
\[
\Phi_4 = L\circ \Phi_3
\]
also satisfies
\begin{equation}\label{we_can_sup_1:eq}
\Phi_4\left(P_{(\cos t, \nu\sin t)}\right)\in\left\{ P_{(\cos t, \nu\sin t)},P_{(\sin t, \nu\cos t)}\right\} \quad \left(\nu\in\T, t\in\left]0,\frac{\pi}{2}\right[\right).
\end{equation}
and $\Phi_4(P_{(1,0)}) = P_{(1,0)}$. We are going to show that
\begin{equation}\label{E:formF1}
\Phi_4\left(P_{(\cos t, \nu\sin t)}\right) = P_{(\cos t, \nu\sin t)}\quad\left(\nu\in\T,\ t\in\left]0,\tfrac{\pi}{2}\right[\right)
\end{equation}
or
\begin{equation}\label{E:formF2}
\Phi_4\left(P_{(\cos t, \nu\sin t)}\right)=P_{(\sin t, \nu\cos t)}\quad\left(\nu\in\T,\ t\in\left]0,\tfrac{\pi}{2}\right[\right)
\end{equation}
holds. Suppose the contrary, i.~e.~there are numbers $\mu,\nu\in\T$ and $t,s\in\left]0,\frac{\pi}{2}\right[\setminus\{\frac{\pi}{4}\}$ such that
\[
\Phi_4(P_{(\cos t, \nu\sin t)}) = P_{(\cos t, \nu\sin t)}
\]
and
\[
\Phi_4(P_{(\cos s, \mu\sin s)}) = P_{(\sin s, \mu\cos s)}.
\]
It follows easily that we can choose these values $\mu,\nu,s,t$ in such a way that at least one of the following possibilities holds:
\begin{itemize}
\item[(I)] $\mu=\nu$, or
\item[(II)] $s=t$ and $|1-\overline{\nu}\mu|<\sqrt{2}$. 
\end{itemize}
Let us make such a choice and consider these possibilities separately.

If we have (I), then the equalities
\[
\Phi_4(P_{(\cos t, \nu\sin t)}) = P_{(\cos t, \nu\sin t)}
\]
and
\[
\Phi_4(P_{(\cos s, \nu\sin s)}) = P_{(\sin s, \nu\cos s)}
\]
hold. Since
\[
\tr P_{(\cos t, \nu\sin t)}P_{(\cos s, \nu\sin s)} = |\cos t \cos s + \sin t \sin s|^2 = \cos^2(t-s)
\]
and
\[
\tr P_{(\cos t, \nu\sin t)}P_{(\sin s, \nu\cos s)} = |\cos t\sin s + \sin t\cos s|^2 = \sin^2(t+s),
\]
we get that either $\cos^2(t-s) = \sin^2(t+s)$ or $\cos^2(t-s) + \sin^2(t+s) = 1$ has to be satisfied by \eqref{E:*}. On the contrary, none of them can be true whenever $t,s\in\left]0,\frac{\pi}{2}\right[\setminus\{\frac{\pi}{4}\}$, hence (I) is impossible.

If (II) happens, then we have
\[
\Phi_4(P_{(\cos t, \mu\sin t)}) = P_{(\cos t, \mu\sin t)},
\]
\[
\Phi_4(P_{(\cos t, \nu\sin t)}) = P_{(\sin t, \nu\cos t)},
\]
and $|1-\mu\overline{\nu}| < \sqrt{2}$. By a straightforward calculation we get
\[
\tr P_{(\cos t, \mu\sin t)}P_{(\cos t, \nu\sin t)} = |\cos^2t + \mu\overline{\nu}\sin^2t|^2,
\]
\[
\tr P_{(\cos t, \mu\sin t)}P_{(\sin t, \nu\cos t)} = \cos^2 t\sin^2 t|1 + \mu\overline{\nu}|^2.
\]
Then we conclude that
\begin{equation}\label{E:trig1}
|\cos^2t + \mu\overline{\nu}\sin^2t|^2 = \cos^2 t\sin^2 t|1 + \mu\overline{\nu}|^2
\end{equation}
or
\begin{equation}\label{E:trig2}
|\cos^2t + \mu\overline{\nu}\sin^2t|^2 + \cos^2 t\sin^2 t|1 + \mu\overline{\nu}|^2 = 1
\end{equation}
holds. On the one hand, the division of equation \eqref{E:trig1} by $\cos^2t\sin^2t$ and a few further calculation gives us
\[
\ctg^2 t + \tg^2 t = 2,
\]
which cannot hold whenever $t\in\left]0,\frac{\pi}{2}\right[\setminus\{\frac{\pi}{4}\}$ is valid. On the other hand, equation \eqref{E:trig2} is equivalent to the following:
\[
\cos^4 t + \sin^4 t + 2\cos^2 t\sin^2 t\re\mu\overline{\nu}+\cos^2 t\sin^2 t (2 + 2\re\mu\overline{\nu}) = 1.
\]
This can be written as
\[
\cos^2 t\sin^2 t\re\mu\overline{\nu} = 0,
\]
which contradicts the inequality $|1-\mu\overline{\nu}|<\sqrt{2}$ and the relation $t\in\left]0,\frac{\pi}{2}\right[\setminus\{\frac{\pi}{4}\}$.

From the above observations we conclude that indeed, \eqref{E:formF1} or \eqref{E:formF2} holds. For any $P\in\irP_1(\C^2)$ we denote the matrix obtained from $P$ by interchanging its diagonal elements by $P^+$. A straightforward computation shows that we have
\[
\Phi_4(P) = P \quad(P\in\irP_1(\C^2)\setminus\{P_{(1,0)},P_{(0,1)}\})
\]
or
\[
\Phi_4(P) = \overline{\left(\begin{matrix}
0 & 1 \\
1 & 0 \\
\end{matrix}\right) P \left(\begin{matrix}
0 & 1 \\
1 & 0 \\
\end{matrix}\right)^*} \quad(P\in\irP_1(\C^2)\setminus\{P_{(1,0)},P_{(0,1)}\}).
\]
Now, we obtain easily that
\[
\Phi_4(P)=UPU^*\quad(P\in\irP_1(\C^2)\setminus\{P_{(1,0)},P_{(0,1)}\})
\]
is valid where $U$ is a unitary or an antiunitary operator on $\C^2$ and $\Phi_4(P_{(1,0)}) = P_{(1,0)}$. Transforming back to $\Phi_3$, we easily complete the proof.
\end{proof}

The rest of the argument concerns the non-injective case, so from now on $\phi\colon\irP_1(\C^2)\to\irP_1(\C^2)$ will be an arbitrary map which satisfies \eqref{E:*}. We are going to define an injective transformation $\Phi\colon\irP_1(\C^2)\to\irP_1(\C^2)$, by composing $\phi$ and an LPM, which also has the property \eqref{E:*}. In order to do this, we need the following observation. Let $P\in\irP_1(\C^2)$ be a projection. Since $\phi$ preserves commutativity in both directions, if $\phi(P) = \phi\left(P^{\bot}\right)$ for some $P\in\irP_1(\C^2)$, then $\phi(P)^{\bot}\notin\rng\phi$. Now, we construct the mentioned map $\Phi$ as follows. If $\phi(P) \neq \phi\left(P^{\bot}\right)$, then $\Phi(P) := \phi(P)$; otherwise let us choose the values of $\Phi$ at $P$ and $P^{\bot}$ such that $\left\{\Phi(P),\Phi(P^{\bot})\right\} = \left\{\phi(P),\phi(P)^{\bot}\right\}$ holds. Trivially, there exists an LPM $L_1$ for which $\Phi =L_1\circ \phi$ holds. It follows that $\phi = L_2\circ \Phi$ for some LPM $L_2$. By the construction of $\Phi$, it is injective and satisfies \eqref{E:*}. Hence Claim \ref{C:3} applies to $\Phi$ and we obtain that
\[
\phi(P) = L_2(UL_3(P)U^*)\in\{UPU^*,UP^{\bot}U^*\}\quad(P\in\irP_1(\C^2)),
\]
where $L_3$ is an LPM and $U$ is a unitary or an antiunitary operator on $\C^2$. This shows that $\phi$ is of the desired form and the proof is complete.
\end{proof}

The next corollary is an obvious reformulation of Theorem \ref{T:proj}. We omit its verification because we only have to use the fact that there is a natural correspondence between rank-one projections and the one-dimensional subspaces of $\irH$, and some observations made just before the proof of Theorem \ref{T:proj}.

\begin{cor}\label{C:proj}
Let $\mathbb{S}$ denote the set of all unit vectors in $\C^2$, and suppose that $\phi\colon\mathbb{S}\to\mathbb{S}$ is such a transformation which satisfies
\[
|\langle \phi(u),\phi(v)\rangle|\in\{|\langle u,v\rangle|, \sqrt{1-|\langle u,v\rangle|^2}\} \quad (u,v\in\mathbb{S}).
\]
Then there exist a unitary or an antiunitary operator $U$ on $\irH$ and a function $f\colon\mathbb{S}\to\T$ such that for each $u\in \mathbb{S}$ we have
\[
\phi(u) \in \{f(u)Uu, f(u)Uu^\perp\}
\]
where $u^\perp$ is an arbitrary unit vector which is orthogonal to $u$.
\end{cor}

After that we present the counterpart of Theorem \ref{T:proj} in higher dimensions which concerns only bijective mappings.

\begin{prop}\label{P:proj3}
Assume that $\dim \irH>2$ and let $\phi\colon \irP_1(\irH)\to \irP_1(\irH)$ be a bijection such that
\[
\big|\big|\big|[\phi(P),\phi(Q)]\big|\big|\big|=\big|\big|\big|[P,Q]\big|\big|\big|\quad(P,Q\in \irP_1(\irH)).
\]
Then there is a unitary or an antiunitary operator $U$ on $\irH$ such that
\begin{equation} \label{E:proj3}
\phi(P)=UPU^*\quad(P\in \irP_1(\irH)).
\end{equation}
\end{prop}

\begin{proof}
Observe that $\phi$ preserves commutativity in both directions. It is easy to check that two elements of $\irP_1(\irH)$ commute if and only if they are identical or orthogonal. We conclude that $\phi$ is such a bijection that preserves orthogonality in both directions, and hence a famous theorem of Uhlhorn (see \cite{U}) implies the existence of either a unitary or an antiunitary operator $U$ on $\irH$ such that \eqref{E:proj3} is fulfilled.
\end{proof}


In the next part of this section we present the proof of Theorem \ref{T:unit}. In order to do this, we will need the following assertion.

\begin{lem}\label{L:unit}
Let $V_1,V_2\in\irU(\irH)$ be such that
\begin{equation}\label{E:mc}
\big|\big|\big|[V_1,P]\big|\big|\big|=\big|\big|\big|[V_2,P]\big|\big|\big| \quad (P\in \irP_1(\irH)).
\end{equation}
Then there is a number $z\in\mathbb{T}$ for which $V_2=zV_1$ or $V_2=zV_1^*$ holds.
\end{lem}

\begin{proof}
Let $V\in\irU(\irH)$ be an operator and $x\in\irH$ be a unit vector. Then we have
\begin{equation}\label{mcp}
\big|\big|\big|[V,x\otimes x]\big|\big|\big|=c\sqrt{1-|\langle Vx,x\rangle|^2}
\end{equation}
where $c$ is the number from \eqref{E:c}. In order to verify \eqref{mcp}, first observe that since $|||.|||$ is unitarily invariant we have
\[
\big|\big|\big|[V,x\otimes x]\big|\big|\big|=|||V(x\otimes x)V^*-x\otimes x|||=|||Vx\otimes Vx-x\otimes x|||.
\]
Let $T := Vx\otimes Vx-x\otimes x$. If $Vx = \xi x$ holds with some $\xi\in\T$, then both sides of \eqref{mcp} is zero, because $\|x\|=1$. Otherwise, it is apparent that $\{x,Vx\}$ is a basis in $\rng T$. A straightforward calculation gives that with respect to this basis, the matrix of $T|_{\ran T}$ is precisely
\[
\left(\begin{array}{cc}
-1 & -\langle Vx,x\rangle \\
\langle x,Vx\rangle & 1
\end{array}\right).
\]
A rather elementary computation gives us $\sigma(T)\setminus\{0\} = \{\pm\sqrt{1-|\langle Vx,x\rangle|^2}\}$. By the spectral theorem and \eqref{E:c}, we deduce that \eqref{mcp} is satisfied.

Now, \eqref{mcp} and the condition \eqref{E:mc} in Lemma \ref{L:unit} easily yield that
\[
|\langle V_1x,x\rangle|^2=|\langle V_2x,x\rangle|^2
\]
is fulfilled for any unit vector and thus for each element $x\in\irH$ and this implies the following:
\[
\langle V_1x,x\rangle\langle x,V_1x\rangle=\langle V_2x,x\rangle\langle x,V_2x\rangle\quad(x\in\irH).
\]
Let $\zeta\in\mathbb{C}$ and $u,v\in\irH$ be any elements and put $x=\zeta u+v$ into this equality. The sides of the obtained equation can be written in the forms $p(\zeta,\overline{\zeta})$ and $q(\zeta,\overline{\zeta})$, respectively with some complex polynomials $p,q$ of two variables. It follows that $(p-q)(\zeta,\overline{\zeta}) = 0$ for all numbers $\zeta\in\C$. We know from \cite[p. 3860]{MT} that in this case all coefficients of $p-q$ has to be 0. Specifically, the coefficient of $\zeta^2$ vanishes which is precisely $\langle V_1u,v\rangle\langle u,V_1v\rangle-\langle V_2u,v\rangle\langle u,V_2v\rangle$. Therefore we have:
\[
\langle V_1u,v\rangle\langle V_1^*u,v\rangle=\langle V_2u,v\rangle\langle V_2^*u,v\rangle.
\]
Now let $w\in\irH$ and put $u=V_1^*w$ in this equality in order to obtain
\[
\langle w,v\rangle\langle (V_1^*)^2w,v\rangle=\langle V_2V_1^*w,v\rangle\langle V_2^*V_1^*w,v\rangle \quad (v,w\in\irH).
\]
Then by \cite[Lemma]{MT} we conclude that $V_2V_1^*$ or $V_2^*V_1^*$ is a scalar operator, therefore there is a number $z\in\mathbb{T}$ such that $V_2\in \{z V_1,z V_1^*\}$.
\end{proof}

Now, we are in a position to prove our statements concerning maps on $\irU(\irH)$. The symbol $\diam(\cdot)$ will denote the diameter of subsets of $\C$.

\begin{proof}[Proof of Theorem \ref{T:unit}]
We consider the two-dimensional case first and then the higher dimensional one.

\smallskip\smallskip

\noindent CASE I. \emph{When $\dim\irH = 2$.} Since $\phi$ is onto, it clearly leaves the quantity
\[
\theta(V)=\sup\{\big|\big|\big|[V,W]\big|\big|\big|\colon W\in\irU(\C^2)\}\quad(V\in\irU(\C^2))
 \]
invariant. In order to obtain a formula for this supremum let $V,W\in\irU(\C^2)$. Let the spectral decomposition of $W$ be $\alpha P+\beta P^{\bot} = (\alpha-\beta)P+\beta I$ with some $\alpha,\beta\in\mathbb{T}$ and $P\in \irP_1(\C^2)$. This implies $\big|\big|\big|[V,W]\big|\big|\big|=\diam\sigma(W)\big|\big|\big|[V,P]\big|\big|\big|$. Now let $x$ be a unit vector in $\rng P$. By \eqref{mcp} we obtain
\[
\big|\big|\big|[V,W]\big|\big|\big| = c\cdot\diam\sigma(W)\sqrt{1-|\langle Vx,x\rangle|^2}.
\]
Clearly as $W$ varies, the number $\diam\sigma(W)$ and the vector $x$ may vary independently. Moreover, when $W$ runs through $\irU(\C^2)$, the quantity $\diam\sigma(W)$ runs through $[0,2]$. Thus, using the last formula we easily conclude that
\[
\theta(V) = 2c\sqrt{1-(\inf\{|\langle Vx,x\rangle|\colon x\in\C^2,\ \Vert x\Vert=1\})^2}.
\]
Now let $\lambda_1,\lambda_2$ be the eigenvalues of $V$ counted according to their multiplicities. By the spectral theorem we obtain
\[
\inf\{|\langle Vx,x\rangle|\colon x\in\C^2,\ \Vert x\Vert=1\}=\inf\{|t\lambda_1+(1-t)\lambda_2|\colon t\in[0,1]\}.
\]
Since $\lambda_1,\lambda_2\in\mathbb{T}$, we get
\[
\inf\{|\langle Vx,x\rangle|\colon x\in\C^2,\ \Vert x\Vert=1\} = |\lambda_1+\lambda_2|/2 = |\tr V|/2,
\]
and therefore
\[
\theta(V) = c\sqrt{4-|\tr V|^2}.
\]

Since $\theta(V) = \theta(\phi(V))$ holds for every $V\in\irU(\C^2)$, we conclude that $|\tr\phi(V)| = |\tr V|$. In particular, $\phi$ leaves the set $U_0(\C^2) = \{V\in\irU(\C^2)\colon\tr V=0\}$ invariant. Since
\[
U_0(\C^2)=\{\xi(2P-I)\ |\ P\in \irP_1(\C^2),\ \xi\in\mathbb{T}\},
\]
we can find maps $\psi\colon \irP_1(\C^2)\to \irP_1(\C^2),\ \zeta\colon \irP_1(\C^2)\to\mathbb{T}$ such that
\[
\phi(2P-I)=\zeta(P)(2\psi(P)-I)\quad(P\in \irP_1(\C^2)).
\]
By \eqref{E:pres} we conclude that
\[
\big|\big|\big|[\psi(P),\psi(Q)]\big|\big|\big|=\big|\big|\big|[P,Q]\big|\big|\big|\quad(P,Q\in \irP_1(\C^2)).
\]
Now, applying Theorem \ref{T:proj} we get
\[
\psi(P)\in \{UPU^*,UP^{\bot}U^*\}\quad(P\in \irP_1(\C^2))
\]
where $U$ is either a unitary or an antiunitary operator on $\C^2$.

Finally, let us define the mapping
\[
\Phi\colon\irU(\C^2)\to\irU(\C^2), \; \Phi(V)=U^*\phi(V)U\ (V\in\irU(\C^2)).
\]
It is quite easy to see that
\begin{equation}\label{E:unitd2last}
\big|\big|\big|[\Phi(V),\Phi(W)]\big|\big|\big|=\big|\big|\big|[V,W]\big|\big|\big|\quad(V,W\in\irU(\C^2))
\end{equation}
and that for any $P\in \irP_1(\C^2)$ we have $\Phi(2P-I)=(\pm\zeta(P))(2P-I)$. Let $V\in\irU(\C^2)$ be a fixed operator and $P$ be an arbitrary rank-one projection on $\C^2$. Putting $W=2P-I$ into \eqref{E:unitd2last}, we obtain the following:
\[
\big|\big|\big|[\Phi(V),P]\big|\big|\big|=\big|\big|\big|[V,P]\big|\big|\big|.
\]
Since this holds for every element $P$ of $\irP_1(\C^2)$, Lemma \ref{L:unit} implies that there exists a number $z\in\mathbb{T}$ such that $\Phi(V) = z V$ or $\Phi(V) = z V^*$. The latter equality is satisfied by all $V\in\irU(\C^2)$, thus transforming back to $\phi$ we conclude that it is of the form appearing in Theorem \ref{T:unit}.

\smallskip\smallskip

\noindent CASE II. \emph{When $\aleph_0 \geq \dim\irH > 2$.}
We define $\overline{T} = (T^*)^{tr}\ (T\in\irB(\irH))$, where $^{tr}$ denotes the transpose of operators with respect to a fixed orthonormal basis in $\irH$. For every normal operator $N\in\irB(\irH)$ we have
\begin{equation}\label{E:conj}
\big|\big|\big|\overline{N}\big|\big|\big|=|||N|||.
\end{equation}
This equality can be proven easily by using the facts that $N$ is unitarily equivalent to a multiplication operator on an $L^2(\mu)$ space where $\mu$ is a measure on some measurable space $X$, and that $N$ is the product of a unitary and a positive operator (see \cite[Chapter IX.]{Co}).

Now observe that $\phi$ preserves commutativity. Hence, by the results published in \cite{MS} there exists a unitary or an antiunitary operator $U$ on $\irH$ such that for every $V\in\irU(\irH)$ we have a Borel function $f_V\colon\sigma(V)\to\T$ by which $\phi(V)=Uf_V(V)U^*$ is satisfied. We define the transformation $\psi\colon\irU(\irH)\to\irU(\irH)$, $\psi(V)=f_V(V)\ (V\in\irU(\irH))$. Observe that by \eqref{E:conj}, we have
\[
\big|\big|\big|[\psi(V_1),\psi(V_2)]\big|\big|\big|=\big|\big|\big|[V_1,V_2]\big|\big|\big| \quad (V_1,V_2\in\irU(\irH)).
\]

Next, let $P\in \irP_1(\irH)$ be arbitrary. By the definition of $\psi$ there are complex numbers $\alpha,\beta$ such that $\psi(2P-I)=\alpha P+\beta I$, moreover, since $\psi(2P-I)\in\irU(\irH)$, we deduce that $\alpha+\beta,\beta\in\mathbb{T}$. Therefore $|\alpha|\leq 2$. Now let $Q\in \irP_1(\irH)$ with $[P,Q]\ne0$. By the previous observations we have $\psi(2Q-I)=\alpha' Q+\beta' I$ with some complex numbers $\alpha',\beta'$ where $|\alpha'|\leq 2$. Since $\psi$ preserves the norm of the commutator, it follows that $4\big|\big|\big|[P,Q]\big|\big|\big|=(|\alpha\alpha'|)\big|\big|\big|[P,Q]\big|\big|\big|$ and thus $|\alpha||\alpha'|=4$ has to be true. We easily infer that $|\alpha|=2$. Now let $V\in\irU(\irH)$ be fixed. By the discussion above we have
\[
\begin{gathered}
2\big|\big|\big|[V,P]\big|\big|\big|=\big|\big|\big|[V,2P-I]\big|\big|\big|=|\alpha|\cdot\big|\big|\big|[\psi(V),P]\big|\big|\big|=2\big|\big|\big|[\psi(V),P]\big|\big|\big|
\end{gathered}
\]
and hence $\big|\big|\big|[\psi(V),P]\big|\big|\big|=\big|\big|\big|[V,P]\big|\big|\big|$ holds for all elements $P\in\irP_1(\irH)$. Applying Lemma \ref{L:unit} we get that there is a number $z\in\mathbb{T}$ such that $\psi(V)=zV$ or $\psi(V)=zV^*$. Since $V$ was an arbitrary element of $\irU(\irH)$, by transforming back to $\phi$ we easily complete this case.
\end{proof}

We finish this section by verifying our result concerning preservers on the set of density operators.

\begin{proof}[Proof of Theorem \ref{T:dense}]
For every operator $A\in\irS(\irH)$ we define
\[
\omega(A)=\sup\big\{\big|\big|\big|[A,B]\big|\big|\big| \colon B\in\irS(\irH)\big\}.
\]
Since $\phi$ is surjective, we clearly have
\begin{equation}\label{E:omega}
\omega(\phi(A))=\omega(A).
\end{equation}
In a very similar way as in Step 1 of the proof of \cite[Theorem 2]{N} and \cite[Theorem 3]{N} we deduce that
\[
\omega(A)=\sup\{\big|\big|\big|[A,x\otimes x]\big|\big|\big|\colon x\in\irH,\ ||x||=1\},
\]
and by \eqref{mcd} we infer
\[
\omega(A)=c\cdot\sup\{\sqrt{\langle A^2x,x\rangle-
\langle Ax,x\rangle^2}\colon x\in\irH,\ \Vert x\Vert=1\}\quad(A\in\irS(\irH)).
\]
Referring to \cite[Lemma 2.6.5]{Ml1} it follows that
\[
\omega(A)=\frac{c}{2}\cdot\diam\sigma(A)
\]
and using \eqref{E:omega} this yields the following:
\[
\diam\sigma(\phi(A))=\diam\sigma(A)\quad(A\in\irS(\irH)).
\]
Since for each $A\in\irS(\irH)$ we have $A\in \irP_1(\irH)$ exactly when the
diameter of the spectrum of $A$ equals 1, we conclude that $\phi$ preserves the set $\irP_1(\irH)$ in both directions. It follows that if $\dim\irH>2$, then $\phi|_{\irP_1(\irH)}\colon\irP_1(\irH)\to\irP_1(\irH)$ is a bijection. Now, according to the value of $\dim\irH$ we can apply Theorem \ref{T:proj} or Proposition \ref{P:proj3} in order to obtain the following: there exists a unitary or an antiunitary operator $U$ on $\irH$ such that if $\dim\irH = 2$, then
\begin{equation}\label{E:proj2d}
\phi(P)\in\{UPU^*,UP^{\bot}U^*\} \quad (P\in \irP_1(\irH)),
\end{equation}
and if $\dim\irH > 2$, then
\begin{equation}\label{E:projhd}
\phi(P) = UPU^* \quad (P\in \irP_1(\irH)).
\end{equation}

Define the transformation $\psi\colon\irS(\irH)\to\irS(\irH)$ by $\psi(A)=U^*\phi(A)U\ (A\in\irS(\irH))$. Then using \eqref{E:conj} and the fact that the commutator of density operators is normal it is easy to see that
\begin{equation}\label{E:denselast}
\big|\big|\big|[\psi(A),\psi(B)]\big|\big|\big|=\big|\big|\big|[A,B]\big|\big|\big|\quad(A,B\in\irS(\irH)).
\end{equation}
Now let $A\in\irS(\irH)$ be a fixed operator and $x\in\irH$ be an arbitrary unit vector. Setting $B=x\otimes x$ in \eqref{E:denselast} and referring to \eqref{E:proj2d} and \eqref{E:projhd} we conclude that
\[
\big|\big|\big|[\psi(A),x\otimes x]\big|\big|\big|=\big|\big|\big|[A,x\otimes x]\big|\big|\big|.
\]
This, by \eqref{mcd} gives us the following:
\[
\langle\psi(A)^2x,x\rangle-
\langle\psi(A)x,x\rangle^2=\langle A^2x,x\rangle-\langle Ax,x\rangle^2.
\]
Since this holds for an arbitrary unit vector $x\in\irH$, applying \cite[Proposition]{MT} we infer
\[
\psi(A)=\lambda I+\tau A,
\]
with some numbers $\lambda\in\mathbb{C}$ and $\tau\in\{-1,1\}$. As the range of $\psi$ is contained in $\irS(\irH)$, we must have $\psi(A)=A$ in the case when $\dim\irH=\infty$, and $\psi(A) \in \{ A, (2/\dim\irH)I-A\}$ if $\dim\irH<\infty$. Since $A\in\irS(\irH)$ was arbitrary, it follows that $\phi$ can be written in the desired form.
\end{proof}


\section{Final remarks}
In this paper we have given a partial answer to the question of Moln\'ar and Timmermann problem. Although the general question in two dimensions remained open. It seems to be an extremely hard problem.

We have also proved a theorem about a preserver problem concerning rank-one projections on $\C^2$. Its higher dimensional version in the bijective case was an easy consequence of Uhlhorn's celebrated generalization of the famous Wigner theorem. However, we do not know anything about the nonbijective case in higher dimensions. We point out that the characterization of those transformations on $\irP_1(\irH)$ which are not necessarily bijective and preserve orthogonality in both directions is unknown. It is certain that we cannot have a similar conclusion with linear or antilinear isometries instead of unitary or antiunitary operators (as in the non-bijective version of Wigner's theorem, see e.~g.~\cite{Ba,G,Gy,SA}). In fact, there are easy counterexamples of such injective transformations on $\irP_1(\irH)$ which preserve orthogonality in both directions and which are not induced by any linear or conjugate linear isometry. For example let us consider an infinite dimensional and separable Hilbert space $\irH$ with orthonormal basis $\{e_j\}_{j=1}^\infty$, and define the following transformation:
\[
\phi(P_{v}) = \left\{
\begin{matrix}
P_{S v} & \text{if } v\neq e_1\\
P_{\frac{1}{\sqrt{2}}(e_1+e_2)} & \text{if } v = e_1
\end{matrix}
\right.,
\]
where $S\in\irB(\irH)$, $Se_j = e_{j+1}\ (j\in\mathbb{N})$ is the usual unilateral shift operator. This implies that we cannot use the same technique for the nonbijective case in higher dimensions.

Concerning Theorems \ref{T:unit} and \ref{T:dense}, we used the bijectivity condition many times during their verifications. A reasonable question is to ask what happens if we drop this assumption.

Let us finish with posing a question. A reasonable measure of commutativity between invertible operators $A,B$ could be the quantity $|||I-A^{-1}B^{-1}AB|||$ which is the distance between the multiplicative commutator and the identity operator. As far as we know, transformations on a certain subclass of invertible operators which preserve this quantity has never been investigated. However, in our point of view, it is a relevant problem.

\bigskip

\section*{Acknowledgements.}
The authors are grateful to Prof.~L.~Moln\'ar for drawing their attention to some of the problems investigated in the paper. They are extremely grateful to the anonymous referee who did a careful reading and gave them many important suggestions.

This research was supported by the European Union and the State of Hungary, co-financed by the European Social Fund in the framework of T\'AMOP-4.2.4.A/2-11/1-2012-0001 'National Excellence Program'. The authors were also supported by the "Lend\" ulet" Program (LP2012-46/2012) of the Hungarian Academy of Sciences.

The second author was also supported by the Hungarian Scientific Research Fund (OTKA) Reg.No.~K81166 NK81402.

\bigskip

\bibliographystyle{amsplain}

\end{document}